\documentclass{amsart}

\usepackage{amsfonts}
\usepackage{amsmath}
\usepackage{amsthm}
\usepackage{graphicx}
\usepackage{wasysym}
\usepackage{amssymb}
\usepackage{mathrsfs}
\usepackage{amscd}
\usepackage{lscape}
\usepackage{longtable}
\usepackage{rotating}
\usepackage{phonetic}
\usepackage{pstricks}
\usepackage{pst-node}
\usepackage{etex}
\usepackage{float}

\input prepictex
\input pictex
\input postpictex

\newtheorem{lemma}{Lemma}[section]
\newtheorem{thm}[lemma]{Theorem}

\newtheorem{prop}[lemma]{Proposition}
\newtheorem{corol}[lemma]{Corollary}
\newtheorem{conj}[lemma]{Conjecture}

\newtheorem{example}[lemma]{Example}
\newtheorem{defn}[lemma]{Definition}

\newtheorem{claim}[lemma]{Claim}

\title{gamma-vectors of edge subdivisions of the boundary of the cross polytope}
\author{NATALIE AISBETT}
\address{School of Mathematics and Statistics\\
University of Sydney, NSW, 2006\\
Australia}
\email{N.Aisbett@maths.usyd.edu.au}
\date{}

\begin{document}
\maketitle{}
\begin{abstract}
For any flag simplicial complex $\Theta$ obtained by stellar subdividing the boundary of the cross polytope in edges, we define a flag simplicial complex $\Gamma(\Theta)$ (dependent on the sequence of subdivisions) whose $f$-vector is the $\gamma$-vector of $\Theta$. This proves that the $\gamma$-vector of any such simplicial complex satisfies the Frankl-F\"{u}redi-Kalai inequalities, partially solving a conjecture by Nevo and Petersen \cite{np}. We show that when $\Theta$ is the dual simplicial complex to a nestohedron, and the sequence of subdivisions corresponds to a flag ordering as defined in \cite{ai}, that $\Gamma(\Theta)$ is equal to the flag simplical complex defined there.
\end{abstract}

\begin{section}{Introduction}
This paper relates to the theory of face enumeration of simplicial complexes. It gives a partial solution to a conjecture by Nevo and Petersen \cite{np} on flag homology spheres, which are a particular class of simplicial complexes whose definition can be found in \cite{ath}. The conjecture is proven for a sub class of flag homology spheres, namely those that can be obtained by subdividing the boundary of the cross polytope in edges. \\
 
Recall that for a $(d-1)$-dimensional simplicial complex $\Theta$, the \emph{$f$-polynomial} is a polynomial in $\mathbb{Z}[t]$ defined as follows:

$$f(\Theta)(t):= f_0 + f_1t + \cdots +f_{d}t^{d},$$ where $f_i(\Theta)$ is the number of $(i-1)$-dimensional faces of $\Theta$, and $f_0(\Theta) =1$. The \emph{$h$-polynomial} is given by

$$h(\Theta)(t): =(1-t)^{d}f(\Theta)\left(\frac{1}{1-t}\right).$$  When $\Theta$ is a homology sphere $h(\Theta)$ is symmetric (this is known as the Dehn-Sommerville relations) hence it can be written

$$h(\Theta)(t) = \sum_{i=0}^{\lfloor\frac{d}{2}\rfloor}\gamma_it^i(1+t)^{d-2i},$$ for some $\gamma_i \in \mathbb{Z}$. Then the \emph{$\gamma$-polynomial} is given by
$$\gamma(\Theta)(t): = \gamma_0 + \gamma_1t + \cdots +\gamma_{\lfloor \frac{d}{2}\rfloor}t^{\lfloor \frac{d}{2}\rfloor}.$$

\noindent The vectors of coefficients of the $f$-polynomial, $h$-polynomial and $\gamma$-polynomial are known respectively as the \emph{$f$-vector}, \emph{$h$-vector} and \emph{$\gamma$-vector}.\\

If $P$ is a simple $d$-dimensional polytope then the \emph{dual simplicial complex} $\Theta_P$ of $P$ is the boundary complex (of dimension $d-1$) of the polytope that is polar dual to $P$. The dual simplicial complex of a $d$-dimensional cube is the boundary complex of the cross polytope and is denoted $\Sigma_{d-1}$. Shaving a codimension two face of a simple polytope is equivalent to the stellar subdivision in an edge of the dual simplicial complex. The set of simplicial complexes that can be obtained from $\Sigma_{d-1}$ by stellar subdivisions in edges is denoted $sd(\Sigma_{d-1})$. By Corollary \ref{corflag} the simplicial complexes in $sd(\Sigma_{d-1})$ are flag homology spheres. \\

Recall that nestohedra are a broad class of simple polytopes introduced in \cite{po} and \cite{prw}. Volodin \cite{vol} has shown that any $d$-dimensional flag nestohedron can be obtained from shaving codimension two faces of the $d$-dimensional cube. Hence, the set $sd(\Sigma_{d-1})$ includes the dual simplicial complex to $d$-dimenisonal flag nestohedra.\\ 

Gal conjectured

\begin{conj}
\cite[Conjecture 2.1.7]{gal}. If $\Theta$ is a flag homology sphere then $\gamma(\Theta)$ is non negative.
\end{conj}

Gal's conjecture was proved for all $\Theta$ in the class $sd(\Sigma_{d-1})$ by Volodin in \cite[Theorem 9]{vol}. This work was done with polytopes rather than their dual simplicial complexes. Further progress on Gal's conjecture was made in Athanasiadis in \cite{ath}. Note that he also studied subdivisions in $\Sigma_{d-1}$ but of a more general kind.\\

Nevo and Petersen conjectured the following strengthening of Gal's conjecture:

\begin{conj}
\cite[Problem 6.4]{np}. If $\Theta$ is a flag homology sphere then $\gamma(\Theta)$ is the $f$-polynomial of a flag simplicial complex. \label{bigconj}
\end{conj}
They proved this in \cite{np} for the following classes of flag simplicial spheres:

\begin{itemize}
\item $\Theta$ is a Coxeter complex,
\item $\Theta$ is the simplicial complex dual to an associahedron,
\item $\Theta$ is the simplicial complex dual to a cyclohedron,
\item $\Theta$ has $\gamma_1(\Theta) \le 3$,
\end{itemize}

We have shown in \cite{ai} that this conjecture holds for the dual simplicial complex of all flag nestohedera. In this paper we prove Conjecture \ref{bigconj} for all $\Theta \in sd(\Sigma_{d-1})$:

\begin{thm}
Suppose that $\Theta \in sd(\Sigma_{d-1})$. Then there is a flag simplicial complex $\Gamma(\Theta)$ such that $f(\Gamma(\Theta)) = \gamma(\Theta)$.
\end{thm}

The definition of $\Gamma(\Theta)$ given in this paper (Definition \ref{maindefn} below) is dependent upon the sequence of subdivisions of $\Sigma_{d-1}$that produce $\Theta$. We show that this construction coincides with the simplicial complexes defined in \cite{ai} when $\Theta$ is the dual simplicial complex of a flag nestohedron. This is not immediately obvious, since the definition in \cite{ai} used the language of building sets.\\

Frohmader \cite[Theorem 1.1]{fr} showed that the $f$-vector of any flag simplicial complex satisfies the Frankl-F\"{u}redi-Kalai (FFK) inequalities (see \cite{ffk}). So Theorem 1.3 implies that $\gamma(\Theta)$ satisfies the FFK inequalities, for all $\Theta\in sd(\Sigma_{d-1})$. In \cite{npt} Nevo Petersen and Tenner show that the $f$-vector of the barycentric subdivision of a homology sphere satisfies the FFK inequalitites.\\

Here is a summary of the contents of this paper. Section 2 contains preliminary well known definitions. Section 3 contains more specific definitions for this paper, as well as some propositions relating to them. Section 4 contains Theorem \ref{bigthm} which is the main result of the paper. Section 5 compares the flag simplicial complexes defined for $\Theta_P$ where $P$ is a flag nestohedron, to those defined in \cite{ai}. \\

\textbf{Acknowledgements}\\

This paper forms part of my PhD research in the School of Mathematics and Statistics at the University of Sydney. I would like to thank my supervisor Anthony Henderson for his feedback and careful reading.

\end{section}

\begin{section}{Definitions}
A \emph{simplicial complex} $\Theta$ with vertex set $V_{\Theta}$ is a set of subsets of $V_{\Theta}$ such that every singleton $\{v\} \in V_{\Theta}$ belongs to $\Theta$, and if $S \in \Theta$ and $I \subseteq S$ then $I \in \Theta$. Elements in $\Theta$ are called \emph{faces}, and the \emph{dimension} of a face $S$ is equal to $|S|-1$.\\

A simplicial complex is \emph{flag} if for every set $S \subseteq V_{\Theta}$ such that for all $a,b \in S$, $\{a,b\} \in \Theta$ we have $S \in \Theta$. A flag simplicial complex is determined by its underlying graph, since the faces are the cliques in this graph.\\

The \emph{link} of a face $F$ in a simplicial complex $\Theta$, denoted $lk_{\Theta}(F)$ is the following subcomplex of $\Theta$.

$$lk_{\Theta}(F) : = \{G \in \Theta~|~G \cup F \in \Theta,~ G \cap F = \emptyset\}.$$

If the simplicial complex is flag then the link of a face is the induced subcomplex on the set of vertices that are adjacent to every vertex in $F$.\\

The \emph{stellar subdivision}, or \emph{subdivision}, of a simplicial complex $\Theta$ in the face $F$ is the simplicial complex $\Theta'$ given by:
\begin{itemize}
\item $\Theta'$ has vertices $V_{\Theta'} =V_{\Theta} \cup \{s\}$ where $s \not \in V_{\Theta}$,\\
\item $\Theta'$ contains all sets in $\Theta$ that do not include $F$, and does not contain any set $K \in \Theta$ such that $F \subseteq K$,\\
\item $\Theta'$ contains sets $\tau \cup\{s\}$ for all $\tau \in \Theta$ such that $F \not \subseteq \tau$, and $\tau \cup F \in \Theta$.
\end{itemize}

If $F$ is simplex, we denote by $F^{\circ}$ the stellar subdivision of $F$ in the face $F$. If $\Theta_1$ and $\Theta_2$ are simplicial complexes, then the \emph{join} of $\Theta_1$ and $\Theta_2$, denoted $\Theta_1*\Theta_2$ is the simplicial complex on the vertex set $V_{\Theta_1} \cup V_{\Theta_2}$ defined by 

$$\Theta_1*\Theta_2:=\{F_1 \cup F_2~|~F_1 \in \Theta_1,~F_2 \in \Theta_2\}.$$ Simplicial complexes $\Theta_1$ and $\Theta_2$ are equivalent, denoted $\Theta_1 \cong \Theta_2$ if there is an inclusion preserving bijection between their faces.

\begin{claim}(See \cite[End of Section 2]{ath}).
If $\Theta$ is a flag simplicial complex, and we perform a stellar subdivision on an edge $S$ of $\Theta$ to obtain the simplicial complex $\Theta'$, then $\Theta'$ is a flag simplicial complex. \label{flagflag}
\end{claim}

\begin{proof}
Consider a set $L$ of vertices of $\Theta'$ such that any pair of vertices in $L$ is in $\Theta'$. We will show that $L \in \Theta'$.\\

If $s \not \in L$ then every two element subset of $L$ does not include $s$, and so they were all in $\Theta$. This implies that $L$ was in $\Theta$ and since $S \not \subseteq L$ this implies that $L \in \Theta'$.\\

Suppose that $s \in L$. Let $\tau$ denote $L/\{s\}$. Then all two element sets in $\tau$ are all in $\Theta$ so that $\tau \in \Theta$. Since $\{s\} \cup \{v\} \in \Theta'$ for all $v \in \tau$ this implies that $\{v\} \cup S \in \Theta$ for all $v \in \tau$. This implies that $\tau \cup S \in \Theta$ since $\Theta$ is flag, and hence that $L \in \Theta'$. 
\end{proof}

\begin{corol}
The simplicial complexes in $sd(\Sigma_{d-1})$ are flag homology spheres. \label{corflag}
\end{corol}

\begin{proof}
The simplicial complexes in $sd(\Sigma_{d-1})$ are flag by Claim \ref{flagflag}, and they are homology spheres since stellar subdivisions do not change the topoplogy of the simplicial complex. 
\end{proof}

\begin{lemma}(Compare \cite[Corollary 1]{vol}).
Suppose $\Theta'$ is a flag homology sphere obtained from a flag homology sphere $\Theta$ (of dimension $d-1$) by stellar subdividing an edge $S$. Then
$$\gamma(\Theta') -\gamma(\Theta)= t\gamma(lk_{\Theta}(S)).$$ \label{sdlink}
\end{lemma}

\begin{proof}

If we stellar subdivide a face $F$ in a simplicial complex $\Theta$ to obtain $\Theta'$, the change in the $f$-vector is
$$f(\Theta') -f(\Theta) = f(F^{\circ}*lk(F)) - f(F*lk(F)),$$ since the set of faces in $\Theta - \Theta'$ is $F*lk(F)$ and the set of faces in $\Theta'-\Theta$ is $F^{\circ}*lk(F)$.\\

In general for simplicial complexes $A$ and $B$ we have
$$f(A*B) = f(A)f(B).$$ Hence
$$f(\Theta') -f(\Theta) = f(lk_{\Theta}(S))[f(S^{\circ}) -f(S)]$$
$$=f(lk_{\Theta}(S))[1+3t+2t^2-(1+2t+t^2)] = f(lk_{\Theta}(S))[t(1+t)].$$

Then $$h(\Theta') - h(\Theta) = (1-t)^df(lk_{\Theta}(S))(\frac{t}{1-t})[\frac{t}{1-t}(1+\frac{t}{1-t})]$$
$$=(1-t)^df(lk_{\Theta}(S))(\frac{t}{1-t})[\frac{t}{(1-t)^2}]$$
$$=t(1-t)^{d-2}f(lk_{\Theta}(S))(\frac{t}{1-t})$$
$$=th(lk_{\Theta}(S)).$$

So $$\gamma(\Theta')\left(\frac{t}{(1+t)^2}\right) - \gamma(\Theta)\left(\frac{t}{(1+t)^2}\right) = \frac{t}{(1+t)^2}\gamma(lk_{\Theta}(S))\left(\frac{t}{(1+t)^2}\right).$$ The result follows.
\end{proof}

\end{section}

\begin{section}{Subdivision sequences}

For the purposes of this paper, say that a \emph{subdivision sequence} is a sequence of simplicial complexes 
$$(\Theta^0,\Theta^1,...,\Theta^k)$$ where $\Theta^0$ is equivalent to $\Sigma_{d-1}$ for some $d$ and each $\Theta^i$ ($i=1,...,k$) is obtained from $\Theta^{i-1}$ by subdividing an edge. (Not up to equivalence, but literally, so the set of vertices of $\Theta^{i}$ consists of the set of vertices of $\Theta^{i-1}$ together with one new vertex). Note that the edge that gets subdivided is determined by the sequence. Call $\Theta^k$ the \emph{result} of the subdivision sequence. For $i = 1,2,...,k$ we label the unique vertex of $\Theta^i$ that is not contained in $\Theta^{i-1}$ by $w_i$, so that $V_{\Theta^k} - V_{\Theta^0} =\{w_1,w_2,...,w_k\}$. \\

Suppose that $(\Theta^0,...,\Theta^k)$ is a subdivision sequence and that $S = \{s_{a},s_{b}\}$ is the $k$th edge subdivided. Then the faces of $\Theta^k$ are in one of the following five sets:\\

$$\mathcal{F}_1: = \{F \in \Theta^k~|~s_a~\hbox{or}~s_b \in F,~\hbox{and}~ w_k \not \in F \},$$

$$\mathcal{F}_2: = \{F \in \Theta^k~|~s_a~\hbox{or}~s_b \in F,~\hbox{and}~w_k \in F\},$$

$$\mathcal{F}_3: = \{F \in \Theta^k~|~ s_a,~s_b \not \in F, ~\hbox{and}~ w_k \in F\},$$

$$\mathcal{F}_4: =\{F \in \Theta^k~|~s_a,~s_b,~w_k \not \in F,~\hbox{and}~\{w_k\} \in lk_{\Theta^k}(F)\},$$

$$\mathcal{F}_5: =\{F \in \Theta^k~|~s_a,~s_b,~w_k \not \in F, ~\hbox{and}~ \{w_k\} \not \in lk_{\Theta^k}(F)\}.$$

Then it is not too hard to show the following:\\
\begin{itemize}
\item[(1)] If $F \in \mathcal{F}_1$ then $$lk_{\Theta^k}(F) \cong lk_{\Theta^{k-1}}(F).$$ If $s_a \in F$ then the vertex $w_k$ in $lk_{\Theta^k}(F)$ replaces the vertex $s_b$ in $lk_{\Theta^{k-1}}(F)$, and if $s_b \in F$ then the vertex $w_k$ in $lk_{\Theta^k}(F)$ replaces the vertex $s_a$ in $lk_{\Theta^{k-1}}(F)$. Otherwise the links are identical.\\

\item[(2)]If $F \in \mathcal{F}_2$, then $$lk_{\Theta^k}(F) = lk_{\Theta^{k-1}}(F-\{w_k\} \cup \{s_b\})$$ if $s_a \in F$ or $$lk_{\Theta^k}(F) = lk_{\Theta^{k-1}}(F-\{w_k\} \cup \{s_a\})$$ if $s_b \in F$.\\

\item[(3)] If $F \in \mathcal{F}_3$ then $$lk_{\Theta^k}(F) =lk_{\Theta^{k-1}}(F-\{w_k\} \cup S)*\Sigma_0,$$ with $s_a,~s_b$ being the vertices of $\Sigma_0$.\\

\item[(4)]If $F \in \mathcal{F}_4$ then $lk_{\Theta^k}(F)$ is the stellar subdivision of $lk_{\Theta^{k-1}}(F)$ in $S$.\\

\item[(5)]If $F \in \mathcal{F}_5$ then $lk_{\Theta^k}(F) =lk_{\Theta^{k-1}}(F)$.\\ 
\end{itemize}

Given a subdivision sequence $(\Theta^0,...,\Theta^k)$, and a face $F$ of $\Theta^k$, there is an induced subdivision sequence
$$(\Phi^0(F),..., \Phi^{l_{F}}(F))_{(\Theta^0,...,\Theta^k)}$$ where $l_F \le k$ that we describe next, whose result $\Phi^{l_F}(F)$ is the simplicial complex $lk_{\Theta^k}(F)$. If the subdivision sequence $(\Theta^0,...,\Theta^k)$ is clear we abbreviate this to the notation $(\Phi^0(F),..., \Phi^{l_{F}}(F))$. The fact that $lk_{\Theta^k}(F)$ is in $sd(\Sigma_{d-1-|F|})$ can be deduced from the definition of the induced subdivision sequence. \\ 

The definition of the induced subdivision sequence is inductive on $k$. If $k=0$ the subdivision sequence $(\Theta^0)$ consists of a single simplicial complex equivalent to $\Sigma_{d-1}$, so that for all $F \in \Sigma_{d-1}$, $lk_{\Theta^0}(F)$ is equivalent to $\Sigma_{d-1-|F|}$. Hence we define the induced sequence to have no subdivisions and set $\Phi^0(F)=lk_{\Theta^0}(F)$.\\

If $k \ge  1$, we assume by induction on $k$ that there is an induced subdivision sequence $$(\Phi^0(F),...,\Phi^{j_{F}}(F))_{(\Theta^0,...,\Theta^{k-1})}$$ for all faces $F \in \Theta^{k-1}$ whose result is $lk_{\Theta^{k-1}}(F)$. Then for any face $F \in \Theta^k$ we consider which of the five sets $F$ lies in (again we suppose the last edge to be subdivided is $S =\{s_a,s_b\}$). Then the subdivision sequence $(\Phi^0(F),...,\Phi^{l_F}(F))_{(\Theta^0,...,\Theta^k)}$ is defined to be:\\

\begin{itemize}
\item[(1)] If $F \in \mathcal{F}_1$ then $l_{F}=j_{F}$ and the simplicial complexes of the induced subdivision sequence $(\Phi^0(F),...,\Phi^{l_F}(F))_{(\Theta^0,...,\Theta^k)}$ are equivalent to the simplicial complexes of $(\Phi^0(F),..., \Phi^{j_{F}}(F))_{(\Theta^0,...,\Theta^{k-1})}$. The map on the vertices is the identity, except that $w_k$ replaces $s_a$ or $s_b$ if either is contained in the sequence. In this case, since $F \in \Theta^{k-1}$, we are giving $lk_{\Theta^k}(F)$ (up to equivalence) the subdivision sequence that is given for $lk_{\Theta^{k-1}}(F)$. \\

\item[(2)] If $F \in \mathcal{F}_2$ and $s_a \in F$ then $l_F = j_{F-\{w_k\} \cup \{s_b\}}$ and the subdivision sequence 
$$(\Phi^0(F),...,\Phi^{l_F}(F))_{(\Theta^0,...,\Theta^k)}$$ is equal to the subdivision sequence $$(\Phi^{0}(F - \{w_k\} \cup \{s_b\}),...,\Phi^{j_{F-\{w_k\} \cup \{s_b\}}}(F - \{w_k\} \cup \{s_b\}))_{(\Theta^0,...,\Theta^{k-1})}.$$
If $s_b \in F$ then the same statements hold with $s_a$ in place of $s_b$. Recalling that $lk_{\Theta^{k}}(F) = lk_{\Theta^{k-1}}(F - \{w_k\} \cup \{s_b\})$, we see that we are adopting the subdivision sequence of $lk_{\Theta^{k-1}}(F-\{w_k\} \cup \{s_b\})$.  \\ 

\item[(3)]  If $F \in \mathcal{F}_3$ then $l_F = j_{F -\{w_k\} \cup S}$, and $(\Phi^0(F),...,\Phi^{l_F}(F))_{(\Theta^0,...,\Theta^k)}$ is the suspension of the subdivision sequence $(\Phi^0(F-\{w_k\} \cup S),...,\Phi^{j_{F-\{w_k\} \cup S}}(F-\{w_k\} \cup S))_{(\Theta^0,...,\Theta^{k-1})}$, meaning that $\Phi^i(F) = \Phi^i(F - \{w_k\} \cup S)*\Sigma_0$. The vertices of $\Sigma_0$ are labeled $s_a$ and $s_b$. \\

\item[(4)] If $F \in \mathcal{F}_4$, then $l_F= j_{F} +1$, and the first $l_F-1$ simplicial complexes of $(\Phi^0(F),...,\Phi^{l}(F))_{(\Theta^0,...,\Theta^k)}$ are equal to the simplicial complexes of the induced subdivision sequence $(\Phi^0(F),...,\Phi^{j_F}(F))_{(\Theta^0,...,\Theta^{k-1})}$, and $\Phi^{l_F}(F)$ is the subdivision of $\Phi^{l_F-1}(F)$ in the edge $S$. Recall that in this case $lk_{\Theta^{k}}(F)$ is the subdivision of $lk_{\Theta^{k-1}}(F)$ in the edge $S$.\\

\item[(5)] If $F \in \mathcal{F}_5$, then $l_F=j_F$ and $(\Phi^0(F),...,\Phi^{l_F}(F))_{(\Theta^0,..,\Theta^{k})}$ is equal to the subdivision sequence $(\Phi^0(F),...,\Phi^{j_F}(F))_{(\Theta^0,..,\Theta^{k-1})}$.\\ 
\end{itemize}

When $F = \emptyset$ it is obvious by induction on $k$ that the induced subdivision sequence $(\Phi^0(\emptyset),...,\Phi^{k}(\emptyset))$ coincides with the subdivision sequence $(\Theta^0,...,\Theta^k)$, since $\emptyset$ is a face in $\mathcal{F}_4$.\\

Given the above induced subdivision sequence $(\Phi^0(F),...,\Phi^{l_F}(F))_{(\Theta^0,...,\Theta^k)}$ define the sets $$W_{(\Theta^0,...,\Theta^k)}(F): = V_{\Phi^{l_F}(F)} - V_{\Phi^0(F)}.$$
When the subdivision sequence is clear from the contex we denote this set by $W_{\Theta^k}(F)$. We label by $w_{1,F},w_{2,F},...,w_{l_F,F}$ the vertices of $W_{\Theta^k}(F)$ where $w_{i,F}$ for $i = 1,...,l$ is the unique vertex in $\Phi^i(F)$ that is not contained in $\Phi^{i-1}(F)$. With this notation we have $w_{j,\emptyset} = w_j$ for $j =1,2,...,k$. We order the sets $W_{\Theta^k}(F)$, $F \in \Theta^k$, by stipulating that if $i<j$ then $w_{i,F} < w_{j,F}$.\\  

\begin{prop}
Suppose $(\Theta^0,...,\Theta^k)$ is a subdivision sequence. For any face $F \in \Theta^k$, the set $W_{\Theta^k}(F)$ satisfies one of the following relations:\\

\begin{itemize}
\item[(1)] If $F \in \mathcal{F}_1$ and $s_a \in F$ then $W_{\Theta^k}(F)$ is equal to $W_{\Theta^{k-1}}(F)$ except $s_b$ is replaced by $w_k$ if $s_b \in W_{\Theta^{k-1}}(F)$. The ordering of the set $W_{\Theta^k}(F)$ is the same as the ordering of the set $W_{\Theta^{k-1}}(F)$ however the vertex $w_k$ takes the position of $s_b$ if $s_b \in W_{\Theta^{k-1}}(F)$. If $s_b \in F$ then the same statements hold with $s_a$ in place of $s_b$. \\

\item[(2)] If $F \in \mathcal{F}_2$ and $F$ contains $s_a$, then $W_{\Theta^k}(F) = W_{\Theta^{k-1}}(F-\{w_k\} \cup \{s_b\})$, and the ordering of the sets coincide.\\

\item[(3)] If $F \in \mathcal{F}_3$ then $W_{\Theta^k}(F) =W_{\Theta^{k-1}}(F-\{w_k\} \cup S),$ and the ordering of the sets coincide. \\  

\item[(4)] If $F \in \mathcal{F}_4$ then $W_{\Theta^k}(F) =W_{\Theta^{k-1}}(F) \cup \{w_k\}$, and the ordering of $W_{\Theta^k}(F)-\{w_k\}$ coincides with the ordering of $W_{\Theta^{k-1}}(F)$, and $w_k$ is last in the ordering.\\

\item[(5)] If $F \in \mathcal{F}_5$ then $W_{\Theta^k}(F) =W_{\Theta^{k-1}}(F)$, and the ordering of the sets coincide.\label{wsetsdefn}
\end{itemize}
\end{prop}

\begin{proof}
This can be proven easily by induction on $k$, using the definition of the induced subdivision sequence. 
\end{proof}

Let $(\Theta^0,...,\Theta^k)$ be a subdivision sequence where $\Theta^0 = \Sigma_{d-1}$. For any face $F \in \Theta^k$ we define a set of vertices 

$$K_{(\Theta^0,...\Theta^k)}(F).$$ This is abbreviated to $K_{\Theta^k}(F)$ when the subdivision sequence is clear from the context. We let

$$K_{\Theta^k}(F): = \bigcap_{v \in F}K_{\Theta^k}(\{v\}),$$ and for any vertex $v \in \Theta^k$ we define $K_{\Theta^k}(\{v\})$ inductively as follows:\\   
If $k=0$ so that $\{v\} \in \Sigma_{d-1}$, then $K_{\Sigma_{d-1}}(\{v\}) = \emptyset$ for all $\{v\} \in \Sigma_{d-1}$. If $k \ge 1$ then $K_{\Theta^k}(\{v\})$ is given by:\\

\begin{itemize}
\item[(1)] If $\{v\} \in \mathcal{F}_1$ (i.e. $v= s_a$ or $s_b$) or if $\{v\} \in \mathcal{F}_5$ (i.e. $v \not \in \{s_a,s_b,w_k\}$ and $\{v\} \not \in lk_{\Theta^k}(\{w_k\})$) then $K_{\Theta^k}(\{v\}) = K_{\Theta^{k-1}}(\{v\})$.\\

\item[(2)] If $\{v\} \in \mathcal{F}_3$ (i.e. $v =w_k$) then $K_{\Theta^k}(\{w_k\}) = K_{\Theta^{k-1}}(\{s_a\})\cap K_{\Theta^{k-1}}(\{s_b\})$.\\

\item[(3)] If $\{v\} \in \mathcal{F}_4$ (i.e. $v \not \in \{s_a,s_b,w_k\}$ and $\{v\} \in lk_{\Theta^k}(\{w_k\})$) then $K_{\Theta^k}(\{v\}) = K_{\Theta^{k-1}}(\{v\}) \cup\{w_k\}$.\\

\end{itemize}

We can also give an inductive definition of $K_{\Theta^k}(F)$.
\begin{prop}
For any face $F \in \Theta^k$ we have 
\begin{itemize}
\item[(1)] If $F \in \mathcal{F}_1$ then $K_{\Theta^k}(F) = K_{\Theta^{k-1}}(F)$.\\

\item[(2)] If $F \in \mathcal{F}_2$ and $s_a \in F$ then $K_{\Theta^k}(F) = K_{\Theta^{k-1}}(F -\{w_k\} \cup \{s_b\})$ (by symmetry the same statement hold with $s_a$ and $s_b$ swapped).\\

\item[(3)] If $F \in \mathcal{F}_3$ then $K_{\Theta^k}(F) = K_{\Theta^{k-1}}(F -\{w_k\} \cup S)$.\\

\item[(4)] If $F \in \mathcal{F}_4$ then $K_{\Theta^k}(F) = K_{\Theta^{k-1}}(F) \cup\{w_k\}$. \\

\item[(5)] If $F \in \mathcal{F}_5$ then $K_{\Theta^k}(F)= K_{\Theta^{k-1}}(F)$. \label{ksetsdefn}
\end{itemize}

\end{prop}

\begin{proof}
We show that the claim holds in each of the five cases for $F \in \Theta^k$. \\

\begin{itemize}
\item[(1)] If $F \in \mathcal{F}_1$ and $s_a \in F$ then for any $w \in F$ either $K_{\Theta^k}(\{w\}) = K_{\Theta^{k-1}}(\{w\})$ or $K_{\Theta^k}(\{w\}) = K_{\Theta^{k-1}}(\{w\}) \cup \{w_k\}$. Also, $w_k \not \in K_{\Theta^k}(\{s_a\} )$. Therefore $K_{\Theta^k}(F) = K_{\Theta^{k-1}}(F)$. By symmetry the claim holds in this case when $s_b \in F$.  \\

\item[(2)] If $F \in \mathcal{F}_2$ and $s_a \in F$ then 
\begin{align*}
K_{\Theta^k}(F) &=\left(\bigcap_{w \in F - \{s_a,w_k\}}K_{\Theta^k}(\{w\})\right) \cap K_{\Theta^k}(\{w_k\})\cap K_{\Theta^k}(\{s_a\})\\
&= \left(\bigcap_{w \in F - \{s_a,w_k\}}K_{\Theta^{k-1}}(\{w\})\right) \cap K_{\Theta^{k-1}}(\{s_a\})\cap K_{\Theta^{k-1}}(\{s_b\})\\
&=\bigcap_{w \in F-\{w_k\} \cup \{s_b\}}K_{\Theta^{k-1}}(\{w\}).
\end{align*}
The second equality uses the fact that for any $w \in F - \{s_a,w_k\}$ we have $K_{\Theta^k}(\{w\}) = K_{\Theta^{k-1}}(\{w\})$ or $K_{\Theta^k}(\{w\}) = K_{\Theta^{k-1}}(\{w\}) \cup \{w_k\}$ yet $w_k \not \in K_{\Theta^{k}}(\{w_k\})$. By symmetry the claim holds when $s_b \in F$.\\  

\item[(3)] If $F \in \mathcal{F}_3$ then
\begin{align*}
K_{\Theta^k}(F) &= \left(\bigcap_{w \in F-\{w_k\}}K_{\Theta^k}(\{w\})\right) \cap K_{\Theta^k}(\{w_k\})\\
&=\left(\bigcap_{w \in F - \{w_k\}}K_{\Theta^{k-1}}(\{w\})\right) \cap K_{\Theta^{k-1}}(\{s_a\}) \cap K_{\Theta^{k-1}}(\{s_b\})\\
&=\bigcap_{w \in F - \{w_k\} \cup S}K_{\Theta^{k-1}}(\{w\})\\
&=K_{\Theta^{k-1}}(F - \{w_k\} \cup S).
\end{align*}

The second equality uses the fact that for all $w \in F - \{w_k\}$ we have $K_{\Theta^k}(\{w\}) = K_{\Theta^{k-1}}(\{w\}) \cup \{w_k\}$, and that $w_k \not \in K_{\Theta^{k}}(\{w_k\})$. \\

\item[(4)] If $F \in \mathcal{F}_4$ then every vertex $w \in F$ is adjacent to $w_k$ and not equal to $s_a$ or $s_b$, so $K_{\Theta^k}(\{w\})$ is the union of  
$K_{\Theta^{k-1}}(\{w\})$ and $\{w_k\}$. Taking the intersection over all vertices $w$ of $F$ gives the claim immediately.\\ 

\item[(5)] If $F \in \mathcal{F}_5$ then there is some vertex $w \in F$ that is not adjacent to both $s_a$ and $s_b$ so that $w_k \not \in K_{\Theta^k}(F)$. Since for every vertex $w \in F$ either $K_{\Theta^k}(\{w\}) = K_{\Theta^{k-1}}(\{w\})$ or $K_{\Theta^k}(\{w\}) = K_{\Theta^{k-1}}(\{w\}) \cup \{w_k\}$ the claim clearly holds in this case. 
\end{itemize}
\end{proof}

\begin{claim}
Given a subdivision sequence $(\Theta^0,...,\Theta^k)$, for any face $F \in \Theta^k$ we have $$|K_{\Theta^k}(F)| =|W_{\Theta^k}(F)|.$$ \label{bijecclaim}
\end{claim}

\begin{proof}
This is clear by induction noting that in the recursive rules of Propositions \ref{wsetsdefn} and \ref{ksetsdefn} this property is maintained.

\end{proof}
Given a subdivision sequence $(\Theta^0,...,\Theta^k)$, for any $F \in \Theta^k$, the set $K_{\Theta^k}(F)$ is a subset of $W_{\Theta^k}(\emptyset) = V_{\Theta^k} - V_{\Theta^0} = \{w_1,...,w_k\}.$ We define an ordering on the set $K_{\Theta^k}(F)$ where for any $w_i,w_j \in K_{\Theta^k}(F)$ we stipulate that if $i<j$ then $w_i < w_j$. Since Claim \ref{bijecclaim} holds, for any face $F \in \Theta^k$ we define the following order preserving bijection 

$$\phi_{\Theta^k,F}: K_{\Theta^k}(F) \rightarrow W_{\Theta^k}(F).$$

\noindent In the case where $F =\emptyset$ this is the identity map $w_i \mapsto w_{i,\emptyset}$. \\

 the same as the choice in $lk_{\Delta}(F) - \{s\}) \cup \{v_2\}$ (respectivley $lk_{\Delta}(F) - \{s\}) \cup \{v_1\}$) to obtain $W_{\Delta}(F)$.\\

\begin{defn}
Given a subdivision sequence $(\Theta^0,...,\Theta^k)$ define a flag simplicial complex $$\Gamma(\Theta^0,...,\Theta^k)$$ on the vertex set $\{w_1,...,w_k\}$, where the condition for $w_a$ to be adjacent to $w_b$ (for $a<b$) is that $w_a$ belongs to $K_{(\Theta^0,...,\Theta^b)}(\{w_b\})$. (When the subdivision sequence is clear we abbreviate this to $\Gamma(\Theta^k)$). \label{maindefn}
\end{defn}

\begin{example}
Take $\Sigma_{3}$, and label the vertices by $\{\pm \epsilon_1, \pm \epsilon_2, \pm \epsilon_3, \pm \epsilon_4\}$, where $\pm \epsilon_i$ for $i=1,2,3,4$ are a pair of non adjacent vertices. Let the subdivision sequence $(\Sigma_3,\Theta^1,\Theta^2, \Theta^3)$ be obtained by: \\

\noindent Step 1: subdivide the edge $\{\epsilon_1,\epsilon_2\}$, to obtain the new vertex $w_1$.\\
Step 2: subdivide the edge $\{\epsilon_3,\epsilon_4\}$, to obtain the new vertex $w_2$.\\
Step 3: subdivide the edge $\{\epsilon_1,w_2\}$ to obtain the new vertex $w_3$.\\

\noindent Then $K_{\Theta^1}(\{\pm \epsilon_3\}) = K_{\Theta^1}(\{\pm \epsilon_4\}) = \{w_1\}$, and in the other cases $K_{\Theta^1}(\{v\}) = \emptyset$.\\
$K_{\Theta^2}(\{\pm \epsilon_1\}) =K_{\Theta^2}(\{ \pm \epsilon_2\})= K_{\Theta^2}(\{w_1\})=\{w_2\}$, and $K_{\Theta^2}(\{\pm \epsilon_3\}) = K_{\Theta^2}(\{\pm \epsilon_4\}) = K_{\Theta^2}(w_2) =\{w_1\}.$\\
Finally, $K_{\Theta^3}(\{-\epsilon_3\}) = K_{\Theta^3}(\{-\epsilon_4\})=K_{\Theta^3}(\{w_2\}) =\{w_1\}$, $K_{\Theta^3}(\{\pm \epsilon_1\})=\{w_2\}$, $K_{\Theta^3}(\{\epsilon_3\}) = K_{\Theta^3}(\{\epsilon_4\}) = \{w_1,w_3\}$, $K_{\Theta^3}(\{\pm \epsilon_2\}) = K_{\Theta^3}(\{w_1\})=\{w_2,w_3\}$, and $K_{\Theta^3}(\{w_3\}) = \emptyset$.\\

Hence $\Gamma(\Sigma_3,\Theta^1,\Theta^2,\Theta^3)$ is the simplicial complex illustrated in Figure \ref{examplefig}.\\

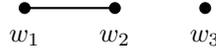
\begin{figure}[H]
\caption{The simplicial complex $\Gamma(\Sigma_3,\Theta^1,\Theta^2, \Theta^3)$.}
\label{examplefig}
\[
\psset{unit=0.8cm}
\begin{pspicture}(5,2)(5.5,3.5)
\qdisk(3,2.5){2.25pt}\qdisk(4.5,2.5){2.25pt}\qdisk(6,2.5){2.25pt}
\rput(3,2){$w_1$}\rput(4.5,2){$w_2$}\rput(6,2){$w_3$}
\psline(3,2.5)(4.5,2.5)
\end{pspicture}
\]
\end{figure}
Note that in this example $K_{\Theta^3}(\{\epsilon_1\}) = \{w_2\}$, whereas $W_{\Theta^3}(\{\epsilon_1\}) = \{w_3\}.$

\end{example}

\end{section}

\begin{section}{The main theorem}

The goal of this section is to prove:
\begin{thm}
For any subdivision sequence  $(\Theta^0,...,\Theta^k)$, $$f(\Gamma(\Theta^0,...,\Theta^k))= \gamma(\Theta^k).$$ \label{bigthm}
\end{thm}

In order to prove this theorem we first need to prove Propositions \ref{bigprop} and \ref{bigproptwo}.

\begin{prop}
Given a subdivision sequence $(\Theta^0,...,\Theta^k)$ and faces $F,~G \in \Theta^k$ such that $G \in lk_{\Theta^k}(F)$, we have that $K_{\Theta^k}(F \cup G) = K_{\Theta^k}(F) \cap K_{\Theta^k}(G)$ maps to $K_{lk_{\Theta^k}(F)}(G)$ under $\phi_{\Theta^k,F}$ ($K_{lk_{\Theta^k}(F)}(G)$ is defined using the induced subdivision sequence with result $lk_{\Theta^k}(F)$). \label{bigprop}
\end{prop}

\begin{proof}
This is a proof by induction on $k$. If $k=0$ then for any face $F \in \Theta^0 = \Sigma_{d-1}$ we have $W_{\Sigma_{d-1}}(F) = \emptyset$ and $K_{\Sigma_{d-1}}(F) = \emptyset$ so that the proposition holds. If $k \ge 1$ then will consider all five cases for faces in $\Theta^k$ and show that the propsition holds in each case. For each case, it is sufficient to assume that $G$ is a vertex $\{g\}$. This is sufficient since if this holds then  $\phi_{\Theta^k,F}$ being a bijection implies that the image of $K_{\Theta^k}(F) \cap K_{\Theta^k}(G) = \bigcap_{w \in G} K_{\Theta^k}(\{w\}) \cap K_{\Theta^k}(F)$ is equal to $\bigcap_{w \in G} K_{lk_{\Theta^k}(F)}(w) = K_{lk_{\Theta^k}(F)}(G)$.\\

\begin{itemize}
\item[(1)] Suppose that $F \in \mathcal{F}_1$, and we may suppose that $s_a \in F$. Recall that $K_{\Theta^k}(F) = K_{\Theta^{k-1}}(F)$, and that either $W_{\Theta^k}(F) = W_{\Theta^{k-1}}(F)$ or $W_{\Theta^k}(F) = (W_{\Theta^{k-1}}(F) - \{s_b\} )\cup \{w_k\}$ where $w_k$ takes the position of $s_b$ in the order. Then $\phi_{\Theta^k,F}$ is the same as $\phi_{\Theta^{k-1},F}$ except for the possible replacement of $s_b$ by $w_k$ in the codomain. \\

\begin{figure}[H]
\caption{The sets described in the case that $g \ne w_k$. Note that $w_k$ and $s_b$ might not be contained in the sets, and they may be contained in $K_{lk_{\Theta^{k}}(F)}(\{g\})$ and $K_{lk_{\Theta^{k-1}}(F)}(\{g\})$}
\label{phimap}

\[
\psset{unit=0.8cm}
\begin{pspicture}(3.2,-3)(8.2,6)
\pscircle(4,3.4){1.4}\pscircle(4,3.1){.6}\rput(4,5.4){$W_{\Theta^k}(F)$}\rput(3.2,3.7){$w_k$}
\pscircle(9,3.4){1.4}\pscircle(9,3.1){.6}\rput(9,5.4){$W_{\Theta^{k-1}}(F)$}\rput(8.2,3.7){$s_b$}
\pscircle(4,-1){1.4}\pscircle(4,-1.3){.6}\rput(4,-3){$K_{\Theta^k}(F)$}
\pscircle(9,-1){1.4}\pscircle(9,-1.3){.6}\rput(9,-3){$K_{\Theta^{k-1}}(F)$}

\rput(4,1){$\phi_{\Theta^k,F}$}\rput(9,1){$\phi_{\Theta^{k-1},F}$}
\psline{->}(4,1.2)(4,1.9)\psline{->}(9,1.2)(9,1.9)
\rput(1,-1.8){$K_{\Theta^k}(F \cup \{g\})$}\pscurve{->}(2,-1.5)(3,-1)(4,-1.5)
\rput(12,-1.8){$K_{\Theta^{k-1}}(F \cup \{g\})$}\pscurve{->}(10.6,-1.5)(9.8,-1)(9,-1.5)
\rput(12,2.8){$K_{lk_{\Theta^{k-1}}(F)}(\{g\})$}\pscurve{->}(10.6,3.1)(9.8,3.5)(9,3.1)
\rput(1,2.7){$K_{lk_{\Theta^{k}}(F)}(\{g\})$}\pscurve{->}(2,3.1)(3,3.5)(4,3.1)
\end{pspicture}
\]
\end{figure}
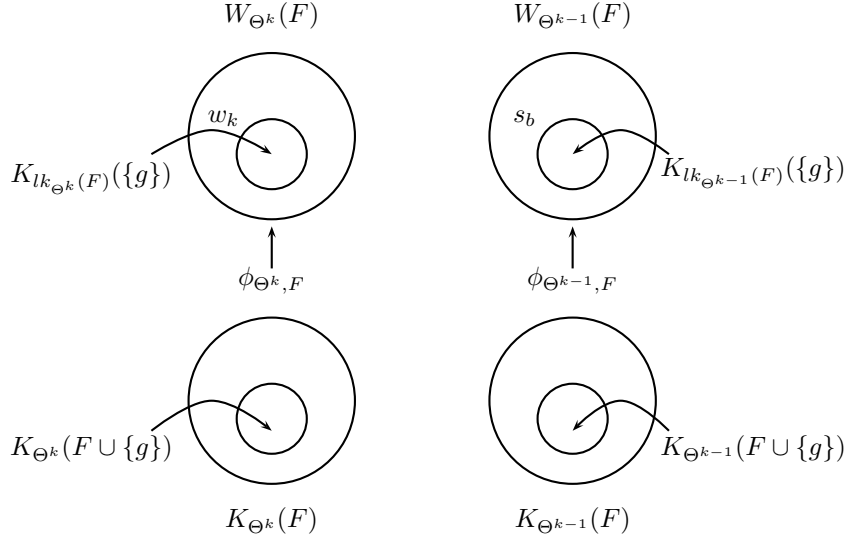

Assume that $g \ne w_k$. Then $K_{\Theta^k}(F \cup \{g\}) = K_{\Theta^{k-1}}(F \cup \{g\})$ (since $F \cup \{g\} \in \mathcal{F}_1$). By the inductive hypothesis $\phi_{\Theta^{k-1},F}(K_{\Theta^{k-1}}(F \cup \{g\})) =K_{lk_{\Theta^{k-1}}(F)}(\{g\})$. By the definition of the induced subdivision sequence we have that $K_{lk_{\Theta^{k}}(F)}(\{g\})$ is equal to $K_{lk_{\Theta^{k-1}}(F)}(\{g\})$ except for the possible replacement of $s_b$ by $w_k$. Hence the proposition holds in this case (see figure \ref{phimap}).\\

Assume that $g = w_k$. Then $K_{\Theta^k}(F \cup \{g\}) = K_{\Theta^{k-1}}(F \cup \{s_b\})$. By the inductive hypothesis $\phi_{\Theta^{k-1},F}(K_{\Theta^{k-1}}(F \cup \{s_b\}))=K_{lk_{\Theta^{k-1}}(F)}(\{s_b\})$. So $\phi_{\Theta^k,F}(K_{\Theta^k}(F \cup \{g\})) = K_{lk_{\Theta^k}(F)}(\{w_k\})$.\\

\item[(2)] Assume that $F \in \mathcal{F}_2$, and we may assume that $s_a \in F$. In this case $lk_{\Theta^k}(F)$ does not contain any of $s_a,~s_b$ or $w_k$ so that $g$ is not equal to any of these vertices. Here $K_{\Theta^k}(F) = K_{\Theta^{k-1}}(F -\{w_k\} \cup \{s_b\})$ and $W_{\Theta^k}(F) = W_{\Theta^{k-1}}(F - \{w_k\} \cup \{s_b\})$, and $\phi_{\Theta^k,F}$ is the same as $\phi_{\Theta^{k-1},F -\{w_k\} \cup \{s_b\}}$. Now $K_{\Theta^k}(F \cup \{g\}) =K_{\Theta^{k-1}}(F - \{w_k\} \cup \{s_b\} \cup \{g\})$ maps under $\phi_{\Theta^{k-1},F - \{w_k\} \cup \{s_b\}}$ to $K_{lk_{\Theta^{k-1}}(F - \{w_k\} \cup \{s_b\})}(\{g\})$ which equals $K_{lk_{\Theta^k}(F)}(\{g\})$ by the definition of the induced subdivision sequence. \\

\item[(3)] Assume that $F \in \mathcal{F}_3$. In this case both $s_a$ and $s_b$ are in $lk_{\Theta^k}(F)$, $K_{\Theta^k}(F) =K_{\Theta^{k-1}}(F -\{w_k\} \cup S)$, $W_{\Theta^{k}}(F) = W_{\Theta^{k-1}}(F- \{w_k\} \cup S)$, and $\phi_{\Theta^{k},F}$ is the same as $\phi_{\Theta^{k-1},F - \{w_k\} \cup S}$.\\

If $g$ is not equal to either $s_a$ or $s_b$ then $K_{\Theta^k}(F \cup \{g\}) = K_{\Theta^{k-1}}(F -\{w_k\} \cup S \cup \{g\}),$ which maps under $\phi_{\Theta^{k-1},F -\{w_k\} \cup S}$ to $K_{lk_{\Theta^{k-1}}(F -\{w_k\} \cup S)}(\{g\})$, and this is equal to $K_{lk_{\Theta^k}(F)}(\{g\})$ by the definition of the induced subdivision sequence.\\

    If $g=s_a$ then $K_{\Theta^k}(F \cup \{s_a\}) = K_{\Theta^{k-1}}(F -\{w_k\} \cup S)$, and this maps under $\phi_{\Theta^{k-1}, F - \{w_k\} \cup S}$ to the set $K_{lk_{\Theta^{k-1}}(F -\{w_k\} \cup S)}(\emptyset) = W_{\Theta^{k-1}}(F - \{w_k\} \cup S)$ which is the set $K_{lk_{\Theta^k}(F)}(\emptyset) = W_{\Theta^k}(F)$. This is the same set as $K_{lk_{\Theta^k}(F)}(\{s_a\})$ since $lk_{\Theta^k}(F)$ is the suspension of $lk_{\Theta^{k-1}}(F - \{w_k\} \cup S)$ in the two additional vertices $s_a$ and $s_b$, which are in $\Sigma_{d-1-|F|}$ in the induced subdivision sequence. By symmetry the result also holds when $g=s_b$. \\

\item[(4)] Suppose that $F \in \mathcal{F}_4$. Then $lk_{\Theta^k}(F)$ is the stellar subdivision of $lk_{\Theta^{k-1}}(F)$ in $S$ and $W_{\Theta^k}(F) = W_{\Theta^{k-1}}(F) \cup \{w_k\}$. We have $K_{\Theta^k}(F) = K_{\Theta^{k-1}}(F) \cup \{w_k\}$ and $\phi_{\Theta^k,F}$ restricts to $\phi_{\Theta^{k-1},F}$ on $K_{\Theta^{k-1}}(F)$ and maps $w_k$ to $w_k$. We now have to consider the different possibilities for $g$.\\

    Suppose that $\{g\} \in \mathcal{F}_1$. We may suppose that $g=s_a$. Then $K_{\Theta^k}(F \cup \{g\}) = K_{\Theta^{k-1}}(F \cup \{g\})$. Under $\phi_{\Theta^{k-1},F}$, $K_{\Theta^{k-1}}(F \cup \{g\})$ maps to the set $K_{lk_{\Theta^{k-1}}(F)}(\{g\})$, and this is equal to $K_{lk_{\Theta^k}(F)}(\{g\})$ since $\{g\} \in \mathcal{F}_1$ in $lk_{\Theta^k}(F)$.\\

    We cannot have $\{g\}$ in $\mathcal{F}_2$ since this implies that $|\{g\}| \ge 2$. \\

    Suppose that $\{g\} \in \mathcal{F}_3$, i.e. that $g=w_k$. In this case $K_{\Theta^k}(F \cup \{w_k\}) = K_{\Theta^{k-1}}(F \cup S)$. Under $\phi_{\Theta^{k-1},F}$ this maps to $K_{lk_{\Theta^{k-1}}(F)}(S)$, and this is equal to $K_{lk_{\Theta^k}(F)}(\{w_k\})$ since $\{w_k\} \in \mathcal{F}_3$ in $lk_{\Theta^k}(F)$.\\
		
		 Suppose that $\{g\} \in \mathcal{F}_4$. Then $K_{\Theta^k}(F \cup \{g\}) = K_{\Theta^{k-1}}(F \cup \{g\}) \cup \{w_k\}$. Now $\phi_{\Theta^{k-1} F}(K_{\Theta^{k-1}}(F \cup \{g\})) =K_{lk_{\Theta^{k-1}}(F)}(\{g\})$, so $\phi_{\Theta^k,F}(K_{\Theta^k}(F \cup \{g\})) = K_{lk_{\Theta^{k-1}}(F)}(\{g\}) \cup \{w_k\} = K_{lk_{\Theta^k}(F)}(\{g\})$, since $\{g\}$ is in $\mathcal{F}_4$ in $lk_{\Theta^k}(F)$.\\

    Suppose that $\{g\} \in \mathcal{F}_5$. Then $K_{\Theta^k}(F \cup \{g\}) = K_{\Theta^{k-1}}(F \cup \{g\})$ and this maps under $\phi_{\Theta^{k-1},F}$ to $K_{lk_{\Theta^{k-1}}(F)}(\{g\})$ which is equal to $K_{lk_{\Theta^k}(F)}(\{g\})$.\\

\item[(5)] Suppose that $F \in \mathcal{F}_5$. Then $K_{\Theta^k}(F) = K_{\Theta^{k-1}}(F)$, $W_{\Theta^k}(F) = W_{\Theta^{k-1}}(F)$, $K_{\Theta^k}(F \cup \{g\}) = K_{\Theta^{k-1}}(F \cup \{g\})$, $\phi_{\Theta^{k},F} = \phi_{\Theta^{k-1},F}$ and $K_{lk_{\Theta^k}(F)}(\{g\}) = K_{lk_{\Theta^{k-1}}(F)}(\{g\})$ so that the proposition clearly holds in this case.
\end{itemize}

\end{proof}

\begin{prop}
Suppose $(\Theta^0,...,\Theta^k)$ is a subdivision sequence. Then for any face $F \in \Theta^k$ the restriction of $\Gamma(\Theta^0,...,\Theta^k)$ to the vertices in $K_{\Theta^k}(F)$ is equivalent to $\Gamma((\Phi^0(F),...,\Phi^{l_F}(F))_{(\Theta^0,...,\Theta^k)})$. The map on the vertices is $\phi_{\Theta^k,F}$.\label{bigproptwo}
\end{prop}

\begin{proof}
We show that the proposition holds by induction on the number of subdivisions. The proposition is clearly true when no subdivisions have been performed. We suppose that the proposition holds for any $\Theta \in sd(\Sigma_{d-1})$ obtained by $k-1$ subdivisions. We let $\Theta^k \in sd(\Sigma_{d-1})$ be obtained by subdividing $\Theta^{k-1}$ in the edge $S = \{s_a,s_b\}$ to give the new vertex $w_k$, and show that the proposition holds for $\Theta^k$. We consider all five cases for a face $F \in \Theta^k$.\\

\begin{itemize}
\item[(1)] Suppose that $F \in \mathcal{F}_1$, and we may suppose that $s_a \in F$. Then $K_{\Theta^k}(F) = K_{\Theta^{k-1}}(F)$, and by the definition of the induced subdivision sequence $$(\Phi^0(F),...,\Phi^{l_F}(F))_{(\Theta^0,...,\Theta^{k})}$$ we have that 
$$\Gamma((\Phi^0(F),...,\Phi^{l_F}(F))_{(\Theta^0,...,\Theta^{k})}) \cong \Gamma((\Phi^0(F),...,\Phi^{l_F}(F))_{(\Theta^0,...,\Theta^{k-1})}),$$ where the map on all vertices is the identity except that $s_b \mapsto w_k$ if $s_b \in W_{\Theta^{k-1}}(F)$. By induction the restriction of $\Gamma(\Theta^{k-1})$ to $K_{\Theta^{k-1}}$ is isomorphic to $\Gamma((\Phi^0(F),...,\Phi^{l_F}(F))_{(\Theta^0,...,\Theta^{k-1})})$.  Hence the proposition holds in this case since $\phi_{\Theta^k,F}$ is the same as $\phi_{\Theta^{k-1},F}$ except for the possible replacement of $s_b$ by $w_k$ in the codomain.\\

\item[(2)] Suppose that $F \in \mathcal{F}_2$. Then $K_{\Theta^k}(F) = K_{\Theta^{k-1}}(F -\{w_k\} \cup \{s_b\})$ and by induction the restriction of $\Gamma(\Theta^{k})$ to $ K_{\Theta^{k-1}}(F -\{w_k\} \cup \{s_b\})$ is equivalent to 
$$\Gamma((\Phi^0(F -\{w_k\} \cup \{w_b\}),...,\Phi^{l_{F-\{w_k\} \cup \{s_b\}}}(F -\{w_k\} \cup \{s_b\}))_{(\Theta^0,...,\Theta^{k-1})}).$$ By the definition of the induced subdivision sequences we have that $$\Gamma((\Phi^0(F),...,\Phi^{l_F}(F))_{(\Theta^0,...,\Theta^{k})})$$ is equal to $$\Gamma((\Phi^0(F - \{w_k\} \cup \{s_b\}),...,\Phi^{l_{F-\{w_k\} \cup \{s_b\}}}(F - \{w_k\} \cup \{s_b\}))_{(\Theta^0,...,\Theta^{k-1})}).$$ Hence the proposition holds in this case.  \\

\item[(3)] Suppose that $F \in \mathcal{F}_3$. Then $K_{\Theta^k}(F) = K_{\Theta^{k-1}}(F-\{w_k\} \cup S)$ and by the definition of the induced subdivision sequences we have that 
$$\Gamma((\Phi^0(F-\{w_k\} \cup S),...,\Phi^{l_{F-\{w_k\} \cup S}}(F-\{w_k\} \cup S))_{(\Theta^0,...,\Theta^{k-1})})$$ is equal to $$\Gamma((\Phi^0(F),...,\Phi^{l_F}(F))_{(\Theta^0,...,\Theta^k)}).$$ Hence the desired condition holds in this case.\\

\item[(4)] Suppose that $F \in \mathcal{F}_4$. Then $K_{\Theta^k}(F) = K_{\Theta^{k-1}}(F) \cup \{w_k\}$, and by the inductive hypothesis the restriction of $\Gamma(\Theta^0,...,\Theta^{k})$ to $K_{\Theta^{k-1}}(F)$ is equivalent to $\Gamma(\Phi^0(F),...,\Phi^{l_{F}-1}(F))_{(\Theta^0,...,\Theta^{k-1})}$. By the definition of the induced subdivision sequence $(\Phi^0(F),...,\Phi^{l_F}(F))_{(\Theta^0,...,\Theta^{k})}$, $\Phi^{l_F}(F)$ is the subdivision of $\Phi^{l_F-1}(F)$ in the edge $S$. Hence $\Gamma((\Phi^0(F),...,\Phi^{l_F}(F))_{(\Theta^0,...,\Theta^{k})})$ is obtained from $\Gamma((\Phi^0(F),...,\Phi^{l_F-1}(F))_{(\Theta^0,...,\Theta^{k-1})})$ by attaching the vertex $w_k$ to the vertices in $K_{(\Phi^0(F),...,\Phi^{l_F}(F))}(\{w_k\}) = K_{(\Phi^0(F),...,\Phi^{l_F-1}(F))}(S)$. The vertex $w_k$ attaches to $K_{\Theta^{k-1}}(S)$ in $\Gamma(\Theta^k)$, and so attaches to the vertices $K_{\Theta^{k-1}}(F) \cap K_{\Theta^{k-1}}(S)$ in $\Gamma((\Phi^0(F),...,\Phi^{l_F-1}(F))_{(\Theta^0,...,\Theta^{k-1})})$. By Proposition \ref{bigprop}, $\phi_{\Theta^{k-1},F}$ maps $K_{\Theta^{k-1}}(F) \cap K_{\Theta^{k-1}}(S)$ to $K_{lk_{\Theta^{k-1}}(F)}(S)$ in $W_{\Theta^{k-1}}(F)$, so that the proposition holds in this case. \\

\item[(5)]  Suppose that $F \in \mathcal{F}_5$. Then $K_{\Theta^k}(F) = K_{\Theta^{k-1}}(F)$, hence it is clear by the definition of the induced subdivision sequences and the relevant sets that the proposition holds in this case. 
\end{itemize}

\end{proof}

\begin{proof}[Proof of Theorem \ref{bigthm}]
We assume by induction that the theorem holds for any simplicial complex in $sd(\Sigma_{i})$ where $i < d-1$. If a subdivision is made on $\Theta \in sd(\Sigma_{d-1})$ in an edge $S$ to obtain $\Theta'$ then by the construction of $\Gamma(\Theta')$, and by Proposition \ref{bigproptwo} we have 

$$f(\Gamma(\Theta')) - f(\Gamma(\Theta)) = tf(\Gamma(lk_{\Theta}(S))),$$ and by the inductive hypothesis we have 

$$f(\Gamma(lk_{\Theta}(S)))=\gamma(lk_{\Theta}(S)).$$ Also, by Proposition \ref{sdlink} 

$$\gamma(\Theta') - \gamma(\Theta) = t\gamma(lk_{\Theta}(S)),$$ so that

$$f(\Gamma(\Theta')) - f(\Gamma(\Theta)) = \gamma(\Theta') - \gamma(\Theta).$$

Since $f(\Gamma(\Sigma_{d-1})) = 1$ and $\gamma(\Sigma_{d-1}) =1$, by induction on the number of subdivions performed the theorem holds. 

\end{proof}

\end{section}

\begin{section}{The dual simplicial complex of nestohedra}

In the paper \cite{ai} for any flag building set $\mathcal{B}$ with respect to a given flag ordering $O=(\mathcal{D},I_1,....,I_k)$ the author defines a flag simplicial complex $\Gamma(O)$, whose face vector is the $\gamma$-vector of the flag nestohedron $P_{\mathcal{B}}$. Adding building set elements in a flag ordering is equivalent to performing edge subdivisions on the dual simplicial complex (see \cite{ai}). In this section we show that $\Gamma(O)$ is equivalent to the flag simplicial complex we define here with respect to that subdivision sequence. First we recall some of the the relevant definitions in \cite{ai}.\\

A \emph{building set} $\mathcal{B}$ on a finite set $S$ is a set of non empty subsets of $S$ such that
\begin{itemize}
\item For any $I,~ J \in \mathcal{B}$ such that $I \cap J \ne \emptyset$, $I \cup J \in \mathcal{B}$.
\item $\mathcal{B}$ contains the singletons $\{i\}$, for all $i \in S$.
\end{itemize}
$\mathcal{B}$ is \emph{connected} if it contains $S$. A building set $\mathcal{B}$ is \emph{flag} if for every non singleton $I \in \mathcal{B}$ there exists $I_1,I_2 \in \mathcal{B}$ such that $I_1 \cap I_2 = \emptyset$ and $I_1 \cup I_2 = I$. Let $\mathcal{B}$ be a building set. A \emph{binary decomposition} or \emph{decomposition} of a non singleton element $I \in \mathcal{B}$ is a set $\mathcal{D} \subseteq \mathcal{B}$ that forms a minimal connected flag building set on $I$.\\

Recall (see \cite{po} and \cite{prw}) the definition of the nestohedron $P_{\mathcal{B}}$ defined for any building set $\mathcal{B}$ . According to \cite{prw} and \cite{nai} the nestohedron $P_{\mathcal{B}}$ is flag exaclty when $\mathcal{B}$ is a flag building set. \\

Suppose that $\mathcal{B}$ is a connected flag building set on $[n]$, $\mathcal{D}$ is a decomposition of $[n]$ in $\mathcal{B}$, and $I_1,I_2,...,I_k$is an ordering of $\mathcal{B}-\mathcal{D}$, such that $\mathcal{B}_j:=\mathcal{D}\cup\{I_1,I_2,...,I_j\}$ is a flag building set for all $j$. (Such an ordering exists by Lemma 6 in \cite{vol}). We call the pair consisting of such a decomposition $\mathcal{D}$ and the ordering on $\mathcal{B} - \mathcal{D}$ a \emph{flag ordering} of $\mathcal{B}$, denoted $O$, or $(\mathcal{D},I_1,...,I_k)$. For any $I_j \in \mathcal{B}- \mathcal{D}$, we say an element in $\mathcal{B}_{j-1}$ is \emph{earlier} in the flag ordering than $I_j$, and an element in $\mathcal{B} - \mathcal{B}_j$ is \emph{later} in the flag ordering than $I_j$.\\

For any $j \in [k]$ define

$$U_j : =\{i~|~i<j, I_i \not \subseteq I_j,~\hbox{there is no}~ I \in \mathcal{B}_{i-1}~\hbox{such that}~I \backslash I_j = I_i \backslash I_j\},$$ and
$$V_j: = \{i~|~i<j,~I_i \subseteq I_j, \exists ~I \in \mathcal{B}_{i-1} ~\hbox{such that}~I_i  \subsetneq I \subsetneq I_j \}.$$ 

Given a flag building set $\mathcal{B}$ with flag ordering $O = (\mathcal{D},I_1,...,I_k)$ define a graph on the vertex set $$V_O=\{v(I_1),...,v(I_k)\},$$where for any $i<j$, $v(I_i)$ is adjacent to $v(I_j)$ if and only if $i \in U_j \cup V_j$. Then define a flag simplicial complex $\Gamma(O)$ whose faces are the cliques in this graph.\\

Suppose that $\Theta_{P_{\mathcal{B}}} \in sd(\Sigma_{d-1})$ is the dual simplical complex to a (flag) nestohedron $P_{\mathcal{B}}$. Suppose also that $\Theta_{P_{\mathcal{B}}}$ has a subdivision sequence $(\Theta^0,...,\Theta^k )$, $\Theta^k = \Theta_{P_{\mathcal{B}}}$, that corresponds to a flag ordering $O = (\mathcal{D},I_1,...,I_k)$ of $\mathcal{B}$. This implies that the vertex $w_i \in \Theta^k$ corresponds to the building set element $I_i$ (this is also the label of the corresponding face of $P_{\mathcal{B}}$). Again we assume that the last edge to be subdivided is $S =\{s_a,s_b\}$. Thus, if $J_a$ is the building set element corresponding to $s_a$ and $J_b$ corresponds to $s_b$ then $J_a \cap J_b = \emptyset$ and $J_a \cup J_b = I_k$.\\

\begin{prop}
Let $\Theta$ be given as above. Then $h \in U_k \cup V_k$ if and only if $w_h \in K_{\Theta^k}(\{w_k\})$. \label{nesto} 
\end{prop}

\begin{proof}
Let $\{J_{m_1},...,J_{m_n}\}$ be the maximal components of the restriction to $I_k$ in $\mathcal{B}_h$, and let $J_{h1},~J_{h2}$ denote the (unique) two elements in $\mathcal{B}_{h-1}$ such that $J_{h1} \cap J_{h2} = \emptyset$ and $J_{h1} \cup J_{h2} = I_h$. First we note that $w_h \in K_{\Theta^k}(\{w_k\})$ is equivalent to $w_h \in K_{\Theta^h}(\{w_{m_l}\})$ for $1 \le l \le n$. This is true since $w_h \in K_{\Theta^h}(\{w_{m_l}\})$ for $1 \le l \le n$ implies that for all of the elements $I_{\beta} \in \mathcal{B}_k$, $\beta > h$ that are subsets of $I_k$, and such that $I_{\beta}$ is a maximal subset of $I_k$ in $\mathcal{B}_{\beta}$, we also have $w_h \in K_{\Theta^{\beta}}(\{w_{\beta}\})$. Conversely, $w_h \not \in K_{\Theta^h}(\{w_{m_l}\})$ for some $l \in \{1,...,n\}$ implies that $w_h \not \in K_{\Theta^{\beta}}(w_{\beta})$ for all $\beta > h$ such that $I_{m_l} \subseteq I_{\beta} \subseteq I_k$ and $I_{\beta}$ is a maximal subset of $I_k$ in $\mathcal{B}_{\beta}$.\\

\begin{itemize}

\item First we suppose that $I_h \subseteq I_k$. We show that $h \in V_k$ if and only if $w_h \in K_{\Theta^k}(\{w_k\})$.\\

Suppose that $h \in V_k$, i.e. $I_h \subseteq I_k$ and there exists a building set element $J_{m_l}$ that is earlier than $I_h$ in the flag ordering such that $I_h \subsetneq J_{m_l} \subsetneq I_k$ (note that it is possible that $J_{m_l}= J_a$ or $J_b$). Then each of the vertices in the set $\{w_{m_1},...,w_{m_n}\}$ are adjacent to both of the vertices $w_{h1}$ and $w_{h2}$, since any pair are a nested set. Thus we have $w_h \in K_{\Theta^k}(\{w_{m_l}\})$ for $1 \le l \le n$, so that $w_h \in K_{\Theta^k}(\{w_k\})$.\\ 

To show that $w_h \in K_{\Theta^k}(\{w_k\})$ implies $h \in V_k$, we show the contrapositive, that $h \not \in V_k$ implies that $w_h \not \in K_{\Theta^k}(\{w_k\})$. $h \not \in V_k$ implies that $I_h = J_{m_l}$ for some $1 \le l \le n$, so that $w_h = w_{m_l} \not \in K_{\Theta^h}(\{w_{m_l}\})$ and (by the reasoning given above) this implies that $w_h \not \in K_{\Theta^k}(\{w_k\})$.\\
 
\item Now suppose that $I_h \not \subseteq I_k$. We show that $h \in U_k$ if and only if $w_h \in K_{\Theta^k}(\{w_k\})$.\\

Suppose that $h \in U_k$. Then $I_h \cap I_k$ is a union of maximal components in the set $J_{m_1},..,J_{m_n}$. Also, each of the maximal components $J_{m_1},...,J_{m_n}$ can intersect at most one of $I_{h1}$ or $I_{h2}$, and cannot be equal to one of $I_{h1}$ or $I_{h2}$ since this implies that $h \not \in U_k$. We therefore have that every $J_{m_l}$, $1 \le l \le n$ is a nested set with either of $I_{h1}$ and $I_{h2}$ since they are a subset of it, or if not a subset of it and their union was in $\mathcal{B}_{h-1}$ then we would not have $h \in U_{k}$. Hence $w_{h1}$ and $w_{h2}$ are adjacent to all of the vertices $w_{m_1},...,w_{m_n}$ in $\Theta^{h-1}$, and therefore $w_h \in K_{\Theta^k}(\{w_k\})$.\\ 

Suppose that $w_h \in K_{\Theta^k}(\{w_k\})$. Then this implies $w_h \in K_{\Theta^{h}}(\{w_{m_l}\})$ for all $1 \le l \le n$, i.e. that $w_{h1}$ and $w_{h2}$ are adjacent to each of $w_{m_l}$ in $\Theta^h$. This implies that neither $I_{h1}$ or $I_{h2}$ are in $\{J_{m_1},...,J_{m_n}\}$ and neither $I_{h1}$ or $I_{h2}$ can be a union of elements in $\{J_{m_1},...,J_{m_n}\}$ (since these are the maximal components). Since each of $I_h,~I_{h1},~I_{h2}$ are a nested set with each of $J_{m_1},...,J_{m_l}$ we have that each of $I_h \cap I_k$, $I_{h1} \cap I_k$ and $I_{h2} \cap I_k$ is a union of elements of $J_{m_1},...,J_{m_l}$. This implies that neither $I_{h1}$ nor $I_{h2}$ is contained in $I_k$. Suppose for a contradiction that $h \not \in U_k$, so that there is an element $I_{\alpha}$ that is earlier than $I_h$ in the flag ordering that has the same image in the contraction by $I_k$ as $I_h$. We suppose that $I_{\alpha}$ is maximal with respect to this property and will consider the following three cases for $I_{\alpha}$.\\

\begin{itemize}
\item Suppose that neither $I_{\alpha} \subseteq I_h$ nor $I_h \subseteq I_{\alpha}$. Then (using the building set axioms) this implies the contradiction that $I_{\alpha}$ is not maximal with this property.\\

\item If $I_{\alpha} \subseteq I_h$ then we have the contradiction that there is an element that is a subset of $I_h$ earlier in the flag ordering that intersects both $I_{h1}$ and $I_{h2}$. \\

\item If $I_h \subseteq I_{\alpha}$ then consider a decomposition of $I_{\alpha}$ in $\mathcal{B}_{\alpha}$. Note that since $I_{\alpha}$ is maximal with this property that $I_{\alpha}$ is the disjoint union of three sections: $I_{h1}, I_{h2}$ and $G: =I_{\alpha}-(I_{h1} \cup I_{h2})$, where $G = \bigcup_{j=1}^s J_{ij}$ is a union of elements in $J_{m_{1}},...,J_{m_n}$. Fix a decompositon $\widetilde{\mathcal{D}}$ of $I_{\alpha}$ in $\mathcal{B}_{\alpha}$. There must be an element $J \in \widetilde{\mathcal{D}}$ that intersects exactly two elements of the set $I_{h1},I_{h2},J_{i1},...,J_{is}$. To find such an element take the set of all elements that intersect more than one of these sets, and from this set choose an element of minimal cardinality. $J$ cannot intersect a pair from $J_{i1},...,J_{is}$ since $J_{m_1},...,J_{m_n}$ are maximal subsets of $I_k$ in $\mathcal{B}_{h}$. $J$ cannot intersect $I_{h1}$ and $I_{h2}$ since this implies that $I_h \in \mathcal{B}_{h-1}$. We cannot have $J$ intersect one of $J_{i1},...,J_{is}$ and one of $I_{h1}$ and $I_{h2}$ since this contradicts the nested set property. Hence we have a contradiction in this case too.\\
\end{itemize}

\end{itemize}

\end{proof}

\begin{corol}
Suppose $\Theta_{P_{\mathcal{B}}}$ is the dual simplicial complex to a flag nestohedron $P_{\mathcal{B}}$, and that the subdivision sequence $(\Theta^0,...,\Theta_{P_{\mathcal{B}}})$ is equivalent to a flag ordering $O$ of the nestohedron. Then $$\Gamma(\Theta^0,...,\Theta_{P_{\mathcal{B}}}) \cong \Gamma(O),$$  where $w_j \mapsto v(I_j)$.
\end{corol}

\begin{proof}
Since Proposition \ref{nesto} holds for all $k$, we have, for any $i,j \in 1,...,k$ such that $i<j$ that $w_i$ is adjacent to $w_j$ in $\Gamma(\Theta_{P_{\mathcal{B}}})$ if and only if $v(I_i)$ is adjacent to $v(I_j)$ in $\Gamma(O)$. 
\end{proof}

\end{section}


\begin{thebibliography}{Trap}



\bibitem{nai} N. Aisbett, \emph{Inequalities between $\gamma$-polynomials of graph-associahedra}, Electronic Journal of Combinatorics, \textbf{Vol. 19}, no. 2, pp. 36.\\

\bibitem{ai} N. Aisbett, \emph{Frankl-F\"{u}redi-Kalai Inequalities on the $\gamma$-vectors of flag nestohedra}, arXiv: 1203.4715v1 [math.CO], 2012.\\

\bibitem{ath} Athanasiadis Christos A, \emph{Flag subdivisions and $\gamma$-vectors}, arXiv: 1106.4520v5 [math.CO], 2012.\\

\bibitem{ffk} P. Frankl, Z. F\"{u}redi, G. Kalai, \emph{Shadows of colored complexes}, Math. Scand. \textbf{Vol. 63}, 1988, pp. 169-178.\\

\bibitem{fr} A. Frohmader, \emph{Face vectors of flag complexes}, Israel J. Math, \textbf{Vol. 164}, 2008, pp. 153-164.\\

\bibitem{gal} S. R. Gal, \emph{Real root conjecture fails for five and higher dimensional spheres}, Discrete Comput. Geom, \textbf{Vol. 34}, 2005, no. 2, pp. 269-284.\\

\bibitem{np} E. Nevo, T. K. Petersen, \emph{On $\gamma$-vectors satisfying the Kruskal-Katona Inequalities}, Discrete Comput. Geom, \textbf{Vol. 45}, 2010, pp. 503-521.\\

\bibitem{npt} E. Nevo, T. K. Petersen, B. E. Tenner, \emph{The $\gamma$-vector of a barycentric subdivision}, J. Combin. Theory Ser. A, \textbf{Vol. 118}, 2011, pp. 1364-1380.\\

\bibitem{po} A. Postnikov, \emph{Permutohedra, associahedra and beyond}, International Mathematics Research Notices, 2009, no. 6, pp. 1026-1106.\\
%
\bibitem{prw} A. Postnikov, V. Reiner and L. Williams, \emph{Faces of generalized permutohedra}, Doc. Math, \textbf{Vol. 13}, 2008, pp. 207-273.\\
%
\bibitem{vol}
V. D. Volodin, \emph{Cubical Realizations of flag nestohedra and a proof of Gal's Conjecture for them}, Uspekhi Mat. Nauk, \textbf{Vol. 65}, 2010, no. 1, pp. 188-190.\\
 
\end{thebibliography}
\end{document}